\documentclass{amsart}
\usepackage{amssymb,latexsym}
\usepackage{amscd,amsthm}
\usepackage{mathrsfs}
\usepackage{hyperref}
\usepackage{marginnote}
\usepackage[usenames,dvipsnames,svgnames,table]{xcolor}
\newtheorem*{introthmA}{Theorem A}
\newtheorem*{introthmB}{Theorem B}
\usepackage{mathptmx} 

\DeclareSymbolFont{usualmathcal}{OMS}{cmsy}{m}{n}
\DeclareSymbolFontAlphabet{\mathcal}{usualmathcal}

\DeclareSymbolFont{letters}{OML}{txmi}{m}{it}

\DeclareMathSymbol{\alpha}{\mathord}{letters}{"0B}
\DeclareMathSymbol{\beta}{\mathord}{letters}{"0C}
\DeclareMathSymbol{\gamma}{\mathord}{letters}{"0D}
\DeclareMathSymbol{\delta}{\mathord}{letters}{"0E}
\DeclareMathSymbol{\epsilon}{\mathord}{letters}{"0F}
\DeclareMathSymbol{\zeta}{\mathord}{letters}{"10}
\DeclareMathSymbol{\eta}{\mathord}{letters}{"11}
\DeclareMathSymbol{\theta}{\mathord}{letters}{"12}
\DeclareMathSymbol{\iota}{\mathord}{letters}{"13}
\DeclareMathSymbol{\kappa}{\mathord}{letters}{"14}
\DeclareMathSymbol{\lambda}{\mathord}{letters}{"15}
\DeclareMathSymbol{\mu}{\mathord}{letters}{"16}
\DeclareMathSymbol{\nu}{\mathord}{letters}{"17}
\DeclareMathSymbol{\xi}{\mathord}{letters}{"18}
\DeclareMathSymbol{\pi}{\mathord}{letters}{"19}
\DeclareMathSymbol{\rho}{\mathord}{letters}{"1A}
\DeclareMathSymbol{\sigma}{\mathord}{letters}{"1B}
\DeclareMathSymbol{\tau}{\mathord}{letters}{"1C}
\DeclareMathSymbol{\upsilon}{\mathord}{letters}{"1D}
\DeclareMathSymbol{\phi}{\mathord}{letters}{"1E}
\DeclareMathSymbol{\chi}{\mathord}{letters}{"1F}
\DeclareMathSymbol{\psi}{\mathord}{letters}{"20}
\DeclareMathSymbol{\omega}{\mathord}{letters}{"21}
\DeclareMathSymbol{\varepsilon}{\mathord}{letters}{"22}
\DeclareMathSymbol{\vartheta}{\mathord}{letters}{"23}
\DeclareMathSymbol{\varpi}{\mathord}{letters}{"24}
\DeclareMathSymbol{\varrho}{\mathord}{letters}{"25}
\DeclareMathSymbol{\varsigma}{\mathord}{letters}{"26}
\DeclareMathSymbol{\varphi}{\mathord}{letters}{"27}

\DeclareMathSymbol{\Gamma}{\mathord}{letters}{"00}
\DeclareMathSymbol{\Delta}{\mathord}{letters}{"01}
\DeclareMathSymbol{\Theta}{\mathord}{letters}{"02}
\DeclareMathSymbol{\Lambda}{\mathord}{letters}{"03}
\DeclareMathSymbol{\Xi}{\mathord}{letters}{"04}
\DeclareMathSymbol{\Pi}{\mathord}{letters}{"05}
\DeclareMathSymbol{\Sigma}{\mathord}{letters}{"06}
\DeclareMathSymbol{\Upsilon}{\mathord}{letters}{"07}
\DeclareMathSymbol{\Phi}{\mathord}{letters}{"08}
\DeclareMathSymbol{\Psi}{\mathord}{letters}{"09}
\DeclareMathSymbol{\Omega}{\mathord}{letters}{"0A}

\DeclareMathSymbol{\upGamma}{\mathalpha}{operators}{"00}
\DeclareMathSymbol{\upDelta}{\mathalpha}{operators}{"01}
\DeclareMathSymbol{\upTheta}{\mathalpha}{operators}{"02}
\DeclareMathSymbol{\upLambda}{\mathalpha}{operators}{"03}
\DeclareMathSymbol{\upXi}{\mathalpha}{operators}{"04}
\DeclareMathSymbol{\upPi}{\mathalpha}{operators}{"05}
\DeclareMathSymbol{\upSigma}{\mathalpha}{operators}{"06}
\DeclareMathSymbol{\upUpsilon}{\mathalpha}{operators}{"07}
\DeclareMathSymbol{\upPhi}{\mathalpha}{operators}{"08}
\DeclareMathSymbol{\upPsi}{\mathalpha}{operators}{"09}
\DeclareMathSymbol{\upOmega}{\mathalpha}{operators}{"0A}

\usepackage[all]{xy} \SelectTips{eu}{} \xyoption{line}
\usepackage{tikz}
\usepackage{tikz-cd}

\newtheorem{theorem}{Theorem}[section]

\newtheorem{lemma}[theorem]{Lemma}

\newtheorem{proposition}[theorem]{Proposition}
\newtheorem{corollary}[theorem]{Corollary}

\theoremstyle{definition}
\newtheorem{definition}[theorem]{Definition}

\newtheorem{remark}[theorem]{Remark}

\newtheorem{example}[theorem]{Example}
\newtheorem*{question}{Question}
\newtheorem{ipg}[theorem]{}

\DeclareMathOperator{\Ext}{Ext}
\DeclareMathOperator{\Hom}{Hom}
\DeclareMathOperator{\Tor}{Tor}

\DeclareMathOperator{\Ch}{Ch}
\DeclareMathOperator{\Ho}{Ho}

\DeclareMathOperator{\FPD}{FPD}
\DeclareMathOperator{\FID}{FID}
\DeclareMathOperator{\FGID}{FGID}
\DeclareMathOperator{\FGPD}{FGPD}

\newcommand{\cat}[1]{\mathcal{#1}}           %% font for categories

\newcommand{\class}[1]{\mathcal{#1}}   %% font for classes

\newcommand{\coker}{\textnormal{coker}\mspace{3mu}}
\newcommand{\pd}[2][\mathcal A]{\textnormal{pd}_{#1}(#2)}

\newcommand{\rep}[1]{\textnormal{Rep}({#1})}

\newcommand{\symbolp}{\pi}
\newcommand{\symboli}{\iota}
\newcommand{\spacep}{\mspace{0.75mu}}
\newcommand{\spacei}{\mspace{1mu}}
\newcommand{\Up}[1][\!\cat{A}]{\cat{U}^{\spacep\symbolp}_{#1}}
\newcommand{\Ui}[1][\!\cat{A}]{\cat{U}^{\spacei\symboli}_{#1}}
\newcommand{\Cp}[1][\!\!\cat{A}]{\cat{C}^{\spacep\symbolp}_{#1}}
\newcommand{\Ci}[1][\!\!\cat{A}]{\cat{C}^{\spacei\symboli}_{#1}}
\newcommand{\Wp}[1][\!\!\cat{A}]{\cat{W}^{\spacep\symbolp}_{#1}}
\newcommand{\Wi}[1][\!\!\cat{A}]{\cat{W}^{\spacei\symboli}_{#1}}
\newcommand{\Fp}[1][\!\!\!\cat{A}]{\cat{F}^{\spacep\symbolp}_{#1}}
\newcommand{\Fi}[1][\!\!\!\cat{A}]{\cat{F}^{\spacei\symboli}_{#1}}

\newcommand{\rightperp}[1]{#1^{\perp}}
\newcommand{\leftperp}[1]{{}^\perp #1}

\newcommand{\GProj}{\operatorname{GProj}}
\newcommand{\GpdA}{\mathrm{Gpd_{\mathcal A}}}

\newcommand{\pdA}{\mathrm{pd_{\mathcal A}}}
\newcommand{\idA}{\mathrm{id_{\mathcal A}}}

\newcommand{\GInj}{\operatorname{GInj}}
\newcommand{\GidA}{\mathrm{Gid_{\mathcal A}}}
\newcommand{\sGProj}{\underline{\mathrm{GPro}}\mathrm{j}}
\newcommand{\sGInj}{\underline{\mathrm{GIn}}\mathrm{j}}

\begin{document}

\title{Quillen equivalences for stable categories}

\author{Georgios Dalezios}
\address{(G.D.) Departamento de Matem\'aticas, Universidad de Murcia, 30100 Murcia, Spain}
\email{georgios.dalezios@um.es}

\author{Sergio Estrada}
\address{(S.E.) Departamento de Matem\'aticas, Universidad de Murcia, 30100 Murcia, Spain}
\email{sestrada@um.es}

\author{Henrik Holm}
\address{(H.H.) Department of Mathematical Sciences, University of Copenhagen, Universitets\-parken 5, 2100 Copenhagen {\O}, Denmark}
\email{holm@math.ku.dk}
\urladdr{http://www.math.ku.dk/\~{}holm/}

\thanks{The first author is supported by the Fundaci\'{o}n S\'{e}neca of Murcia 19880/GERM/15. The second author is supported by the grant MTM2016-77445-P and FEDER funds and the grant 18394/JLI/13 by the Fundaci\'on S\'eneca-Agencia de Ciencia y Tecnolog\'{\i}a de la Regi\'on de Murcia in the framework of III PCTRM 2011-2014 }
\date{}


% 18E10 Exact categories, abelian categories 
% 18E30 Derived categories, triangulated categories 
% 18G55 Homotopical algebra 
% 13D02 Syzygies, resolutions, complexes (for commutative rings)
% 16E05 Syzygies, resolutions, complexes (for general rings)

\begin{abstract}
For an abelian category $\class A$ we investigate when the stable categories $\sGProj(\class A)$ and $\sGInj(\class A)$ are triangulated equivalent. To this end, we realize these stable categories as homotopy categories of certain (non-trivial) model categories and give conditions on $\cat{A}$ that ensure the existence of a Quillen equivalence between the model categories in question. We also study when such a Quillen equivalence transfers from $\class A$ to categories naturally associated to $\class A$, such as $\mathrm{Ch}(\cat{A})$, the category of chain complexes in $\cat{A}$, or $\mathrm{Rep}(Q,\cat{A})$, the category $\cat{A}$-valued representations of a quiver $Q$.
\end{abstract}

\maketitle

\section{Introduction}

Over an Iwanaga--Gorenstein ring $A$, that is, a ring which is noetherian and has finite injective dimension from both sides, the category $\mathrm{MCM}(A)$ of (finitely generated) maximal Cohen--Macaulay $A$-modules\footnote{\,In the important special case where $A$ is a quasi-Frobenius ring, for example, if $A=kG$ is the group algebra of a finite group $G$ with coefficients in a field $k$, the category $\mathrm{MCM}(A)$ is just the category $\mathrm{mod}(A)$ of all finitely generated $A$-modules.} is a Frobenius category in which the \mbox{projective-injective}~objects are precisely the finitely generated projective $A$-modules. The associated stable category $\underline{\mathrm{MCM}}(A)$ is therefore triangulated, and a classic result of Buchweitz \cite[Thm.~4.4.1]{ROB86} shows that $\underline{\mathrm{MCM}}(A)$ is triangulated equivalent to the singularity category\footnote{\,The singularity category \smash{$\mathcal{D}_\mathrm{sg}(A)$} is defined to be the Verdier quotient \smash{$\mathcal{D}^\mathrm{b}(A)/\mathcal{D}^\mathrm{b}_\mathrm{perf}(A)$} of the bounded derived category \smash{$\mathcal{D}^\mathrm{b}(A)$}, whose objets are complexes of $A$-modules with bounded and finitely generated homology, by the subcategory \smash{$\mathcal{D}^\mathrm{b}_\mathrm{perf}(A)$}, whose objects are isomorphic (in \smash{$\mathcal{D}^\mathrm{b}(A)$}) to a perfect complex, that is, to a bounded complex of finitely generated projective $A$-modules. The name \emph{singularity category} and the symbol \smash{$\mathcal{D}_\mathrm{sg}(A)$} seem to be the popular choices nowadays, however, in the work of Buchweitz \cite[Def.~1.2.2]{ROB86}, this category is called the \emph{stabilized derived category} and denoted by \smash{\underline{$\mathcal{D}^\mathrm{b}(A)$}}, and in the work of Orlov \cite{Orlov2004}, it is called the \emph{triangulated category of singularities} and denoted by \smash{$\mathcal{D}_\mathrm{sg}(A)$}.} $\mathcal{D}_\mathrm{sg}(A)$, which is an important matematical object that has been studied by many authors; see \cite{MR1450996,MR2846489,MR3290681,MR2680200}. 

If $A$ is not Iwanaga--Gorenstein, then the category $\mathrm{MCM}(A)$ is, in general,  not Frobenius. However, over any ring $A$ one can always consider the category $\GProj(A)$ of so-called Gorenstein projective modules (which are not assumed to be finitely generated); this category is always Frobenius and the associated stable category $\sGProj(A)$ is triangulated. In the case where $A$ is Iwanaga--Gorenstein, an $A$-module is maximal Cohen--Macaulay if and only if it is finitely generated and Gorenstein projective, and hence $\underline{\mathrm{MCM}}(A)$ can be identified with the finitely generated modules in $\sGProj(A)$. This explains the interest in the category $\sGProj(A)$ for general ring $A$. Its injective counterpart $\sGInj(A)$, the stable category of Gorenstein injective $A$-modules, is equally important and has been studied in e.g \cite{Benson-Krause-kG,HKr05}.

Our work is motivated by a recent result of Zheng and Huang \cite{zheng-huang-triangulated-equivalences} which asserts that for many rings $A$, the categories $\sGProj(A)$ and $\sGInj(A)$ are equivalent as triangulated categories. As it makes sense to consider the stable categories $\sGProj(\cat{A})$ and $\sGInj(\cat{A})$ for any bicomplete abelian category $\cat{A}$ with enough projectives and injectives (see Section~\ref{Preliminaries} for details), the following question naturally arises:
 
\begin{question}
For which abelian categories $\cat{A}$ (assumed to be bicomplete with enough projectives and injectives) are $\sGProj(\cat{A})$ and $\sGInj(\cat{A})$ equivalent as triangulated categories?
\end{question}

Every Frobenius category $\cat{E}$, in particular, $\GProj(\class A)$ and $\GInj(\class A)$, can be equipped with a canonical model structure which has the property that the associated homotopy category $\Ho(\cat{E})$ is equivalent to the stable category $\underline{\cat{E}}$; see e.g.~\cite[Prop.~4.1]{Gillespie2013}. Thus, if the Frobenius categories $\GProj(\class A)$ and $\GInj(\class A)$, equipped with these canonical model structures, happen to be Quillen equivalent, then we get an affirmative answer to the question above. However, the model categories $\GProj(\class A)$ and $\GInj(\class A)$, and even the underlying ordinary categories, will rarely be (Quillen) equivalent. In this paper, we consider instead the categories
\begin{displaymath}
\Up[]= \left\{ M\in \class A \ | \ \GpdA(M)<\infty\right\}
  \quad \text{and} \quad 
\Ui[]=\left\{ N\in \class A \ | \ \GidA(N)<\infty\right\}
\end{displaymath}
and show in Theorems~\ref{thmUp} and \ref{thmUi} that $\Up[]$ and $\Ui[]$ can be equipped with model structures for which the associated homotopy categories  $\Ho(\Up[])$ and $\Ho(\Ui[])$ are the stable categories $\sGProj(\class A)$ and $\sGInj(\class A)$. The advantage of having these realizations of the stable categories is that in several cases the model categories $\Up[]$ and $\Ui[]$ will be Quillen equivalent---even though $\GProj(\class A)$ and $\GInj(\class A)$ are not---and in such cases we therefore get an affirmative answer (for a strong reason) to the question above\footnote{\,In general, we do not expect every (triangulated) equivalence between $\sGProj(\class A)$ and $\sGInj(\class A)$, if such an equivalence even exist, to be induced from a Quillen equivalence between model categories. Indeed, it is well-known~that there are examples of non Quillen equivalent model categories with equivalent homotopy categories.}. To investigate when $\Up[]$~and~$\Ui[]$ will be Quillen equivalent, we introduce the notion of a Sharp--Foxby adjunction (Definition~\ref{Sharp-Foxby}). We prove in Theorem~\ref{Ho} and Corollary~\ref{cor:SF-imples-GProj-GInj} that if $\cat{A}$ ad\-mits such an adjuntion, then $\Up[]$ and $\Ui[]$ will be Quillen equivalent:

\begin{introthmA}
 A Sharp--Foxby adjunction $(S,T)$ on $\class A$ induces a Quillen equivalence between the model categories $\Up[]$ and $\Ui[]$. Thus the total (left/right) derived functors of $S$ and $T$ yield an adjoint equivalence of the corresponding homotopy categories,
\begin{equation*}
  \xymatrix@C=3pc{
     \sGProj(\cat{A}) \simeq \mathrm{Ho}(\Up[]) \ar@<0.7ex>[r]^-{\mathbf{L}S} & 
     \mathrm{Ho}(\Ui[]) \simeq \sGInj(\cat{A}) \ar@<0.7ex>[l]^-{\mathbf{R}T}
  }\!.
\end{equation*} 
In fact, this is an equivalence of triangulated categories.
\end{introthmA}

The choice to work with the categories $\Up[]$ and $\Ui[]$ is historically motivated by classic results in commutative algebra by Sharp \cite{RYS72} and Foxby \cite{HBF72} . In the language of this paper, the results can be phrased as follows: If $A$ is a Cohen--Macaulay ring with a dualizing module $D$, then the functors $S=D\otimes_A-$ and $T=\Hom_A(D,-)$ constitute a Sharp--Foxby adjunction on \mbox{$\cat{A}=\mathrm{Mod}(A)$}; see Example~\ref{Auslander-Bass} for details. Thus, for such rings Theorem~A improves the previously mentioned result of Zheng and Huang \cite{zheng-huang-triangulated-equivalences} to a triangulated equivalence between $\sGProj(A)$ and $\sGInj(A)$ induced by a Quillen equivalence.

In Sections \ref{sec:Ch} and \ref{sec:Rep} we investigate to what extend a Sharp--Foxby adjunction on a category $\cat{A}$ (and hence also a Quillen equivalence between the model categories $\Up[]$ and $\Ui[]$, see Theorem~A) transfers to categories naturally constructed from $\class A$, such as $\mathrm{Ch}(\cat{A})$, the category of chain complexes in $\cat{A}$, or $\mathrm{Rep}(Q,\cat{A})$, the category $\cat{A}$-valued representations of a quiver $Q$. Theorems~\ref{Ch} and \ref{Rep} combined yield the following result.

\begin{introthmB}
Assume that $(S,T)$ is a Sharp--Foxby adjunction on $\cat{A}$; in particular, $\sGProj(\cat{A})$ and $\sGInj(\cat{A})$ are equivalent as triangulated categories by Theorem~A. Assume furthermore that the finitistic projective and the finitistic injective dimensions of $\cat{A}$ are finite.
 
 If $\cat{B}=\mathrm{Ch}(\cat{A})$ or if $\cat{B}=\mathrm{Rep}(Q,\cat{A})$ for a left and right rooted quiver $Q$, then degreewise/vertexwise application of $S$ and $T$ yields a Sharp--Foxby adjunction on $\cat{B}$; in particular, $\sGProj(\cat{B})$ and $\sGInj(\cat{B})$ are  equivalent as triangulated categories.
 \end{introthmB}

\section{Preliminaries}
\label{Preliminaries}

Throughout this paper, $\cat{A}$ denotes any bicomplete abelian category with enough projectives and enough injectives.

Gorenstein projective and Gorenstein injective modules (over any ring) were defined by Enochs and Jenda \cite[\S2]{EEnOJn95b}, but the definition works for objects in any abelian category:

\begin{definition}
  \label{Gorenstein-objects}
  An acyclic (= exact) complex $P = \cdots\to P_1\to P_0\to P_{-1}\to \cdots$ of projective objects in $\cat{A}$ is called {\it totally acyclic} if for any projective object $Q$ in $\cat{A}$ the complex 
\begin{displaymath}  
  \mathrm{Hom}_\cat{A}(P,Q) = \ \cdots\longrightarrow \mathrm{Hom}_\cat{A}(P_{-1},Q)\longrightarrow \mathrm{Hom}_\cat{A}(P_0,Q)\longrightarrow \mathrm{Hom}_\cat{A}(P_{1},Q)\longrightarrow \cdots
\end{displaymath}  
is acyclic. An object $G$ in $\cat{A}$ is called {\it Gorenstein projective} if it is a cycle of such a totally acyclic complex of projectives, that is, if $G=\mathrm{Z}_j(P)$ for some integer $j$. We write~$\GProj(\cat{A})$ for the full subcategory of $\cat{A}$ consisting of all Gorenstein projective objects.

Dually, an acyclic complex $I = \cdots\to I_1\to I_0\to I_{-1}\to \cdots$ of injective objects in $\cat{A}$ is called {\it totally acyclic} if for any injective object $E$ in $\cat{A}$ the complex 
\begin{displaymath}  
  \mathrm{Hom}_\cat{A}(E,I) = \ \cdots\longrightarrow \mathrm{Hom}_\cat{A}(E,I_{1})\longrightarrow \mathrm{Hom}_\cat{A}(E,I_0)\longrightarrow \mathrm{Hom}_\cat{A}(E,I_{-1})\longrightarrow \cdots
\end{displaymath}  
is acyclic. An object $H$ in $\cat{A}$ is called {\it Gorenstein projective} if it is a cycle of such a totally acyclic complex of injectives, that is, if $H=\mathrm{Z}_j(I)$ for some integer $j$. We write~$\GInj(\cat{A})$ for the full subcategory of $\cat{A}$ consisting of all Gorenstein injective objects.

The {\it Gorenstein projective dimension}, $\GpdA(M)$, of an object $M$ in $\cat{A}$ is defined by de\-cla\-ring that one has $\GpdA(M)\leqslant n$ (for $n\in\mathbb{N}_0$) if and only if there exists an exact sequence $0\to G_n\to G_{n-1}\to \cdots\to G_0\to M\to 0$ in $\cat{A}$ with $G_0,\ldots, G_n \in \GProj(\cat{A})$. The {\it Gorenstein injective dimension}, $\GidA(M)$, of $M$ is defined analogously.

\end{definition}

Recall that a \emph{Frobenius category} is an exact category $\cat{E}$ with enough (relative) projectives and enough (relative) injectives and where the classes of projectives and injectives coincide; such objects are called \emph{projective-injective} (or just \emph{pro-injective}) objects. The \emph{stable category} $\underline{\cat{E}}$ is the quotient category
\mbox{$\cat{E}/\mspace{-4mu}\sim$} where the relation ``$\sim$'' is defined by \mbox{$f \sim g$} (here $f$ and $g$ are parallel morphisms in $\cat{E}$) if $f-g$ factors through a projective-injective object. The category $\underline{\cat{E}}$ is triangulated as described in Happel \cite[Chap.~I\S2]{HapT} (see also \ref{triangulated}).

The following result is well-known, but for completeness we include a short proof.

\begin{proposition}  
  \label{GProj-is-Frobenius}
  The category $\GProj(\cat{A})$ is Frobenius and the projec\-tive-injective objects herein are the projective objects in $\cat{A}$. Thus, the stable category $\sGProj(\cat{A})$ is triangulated.
  
  The category $\GInj(\cat{A})$ is Frobenius and the projec\-tive-injective objects herein are the injective objects in $\cat{A}$. Thus, the stable category $\sGInj(\cat{A})$ is triangulated.  
\end{proposition}

\begin{proof}
  We only show the claims about the category $\GProj(\cat{A})$, as the claims about $\GInj(\cat{A})$ are proved similarly. The proof only uses basic properties of Gorenstein projective objects. In the case of modules, that is, if $\cat{A}=\mathrm{Mod}(A)$ for a ring $A$, these properties are recorded in \cite{HHl04a}, however, the reader easily verifies that the same properties hold for Gorenstein projective objects in any abelian category $\cat{A}$ with enough projectives.
  
  First of all, by \cite[Thm.~2.5]{HHl04a} the class $\GProj(\cat{A})$ is an additive extension-closed subcategory of the abelian category $\cat{A}$, and thus $\GProj(\cat{A})$ is an exact category. Clearly, every (categorical) projective object $P$ in $\cat{A}$ is a (relative) projective object in $\GProj(\cat{A})$, but it is also (relative) injective since every short exact sequence $0 \to P \to G \to G' \to 0$ in $\cat{A}$ with $G,G' \in \GProj(\cat{A})$ splits; indeed by \cite[Prop.~2.3]{HHl04a} one has $\Ext^1_{\cat{A}}(G',P)=0$. By the definition of Gorenstein projective objects, every $G \in \GProj(\cat{A})$ fits into short exact sequences $0 \to H \to P \to G \to 0$ and $0 \to G \to P' \to H' \to 0$ in $\cat{A}$ where $P,P'$ are (categorical) projective and $H,H'$ are Gorenstein projective. It follows that if $G$ is (relative) projective or (relative) injective, then $G$ is a direct summand of a (categorical) projective object, $P$ or $P'$, and hence $G$ is (categorical) projective. It also follows that $\GProj(\cat{A})$ has enough (relative) projectives and enough (relative) injectives.
\end{proof}

In Theorems~\ref{thmUp} and \ref{thmUi} we construct certain model categories $\Up[]$ and $\Ui[]$ for which the associated homotopy categories $\Ho(\Up[])$ and $\Ho(\Ui[])$ are $\sGProj(\cat{A})$ and $\sGInj(\cat{A})$.

The standard references for the theory of cotorsion pairs are Enochs and Jenda~\cite{rha} and G{\"o}bel and Trlifaj \cite{GobelTrlifaj}. Below we recall a few notions that we need.

\begin{ipg}
\label{cotorsion-pairs}
A pair $(\class{X},\class{Y})$ of classes of objects in $\cat{A}$ is a \emph{cotorsion pair} if $\rightperp{\class{X}}=\class{Y}$ and $\class{X} = \leftperp{\class{Y}}$. Here, given a class $\class{C}$ of objects in $\cat{A}$, the right orthogonal  $\rightperp{\class{C}}$ is defined to be the class of all $Y \in \cat{A}$ such that $\Ext^1_{\cat{A}}(C,Y) = 0$ for all $C \in \class{C}$. The left orthogonal $\leftperp{\class{C}}$ is defined similarly. A cotorsion pair $(\class{X},\class{Y})$ is \emph{hereditary} if $\Ext^i_{\cat{A}}(X,Y) = 0$ for all $X \in \class{X}$, $Y \in \class{Y}$, and $i \geqslant 1$. A cotorsion pair $(\class{X},\class{Y})$ is \emph{complete} if it has \emph{enough projectives} and \emph{enough injectives}, i.e.~for each $A \in \cat{A}$ there exist short exact sequences $0 \xrightarrow{} Y \xrightarrow{} X \xrightarrow{} A \xrightarrow{} 0$ (enough projectives)
and $0 \xrightarrow{} A \xrightarrow{} Y' \xrightarrow{} X' \xrightarrow{} 0$ (enough injectives)  with $X,X' \in \class{X}$ and $Y,Y' \in \class{Y}$. 

In order for the above to make sense, the category $\cat{A}$ only needs to be exact (not necessarily abelian), so that one has a notion of ``short exact sequences'' (often called \emph{conflations}) and hence also of (Yoneda) $\Ext_\cat{A}$.
\end{ipg}

Cotorsion pairs are related to relative homological algebra, see \cite{rha}, and due to 
work of Hovey \cite{hovey} they are also related to abelian (or exact) model category structures.

\begin{ipg}
\label{Hovey-triple}
An \emph{abelian model structure} on $\cat{A}$, that is, a model structure on $\cat{A}$ which is compatible with the abelian structure in the sense of \cite[Def.~2.1]{hovey}, corresponds by Thm.~2.2 in \emph{loc.~cit.}~to a triple $(\class{C},\class{W},\class{F})$ of classes of objects in $\cat{A}$ for which $\class{W}$ is thick\footnote{\,Recall that a class $\cat{W}$ in an abelian (or, more generally, in an exact) category $\cat{A}$ is \emph{thick} if it is closed under retracts and satisfies that whenever two out of three terms in a short exact sequence are in $\class{W}$, then so is the third.} and $(\class{C} \cap \class{W},\class{F})$ and $(\class{C},\class{W} \cap \class{F})$ are complete cotorsion pairs in $\cat{A}$. Such a triple $(\class{C},\class{W},\class{F})$ is called a \emph{Hovey triple} in $\cat{A}$. In the model structure on $\cat{A}$ determined by such a Hovey triple, $\class{C}$ is precisely the class of cofibrant objects, $\class{F}$ is precisely the class of fibrant objects, and $\class{W}$ is precisely the class of trivial objects (that is, objects weakly equivalent to zero). A \emph{hereditary Hovey triple} is a Hovey triple $(\class{C},\class{W},\class{F})$ for which the associated complete cotorsion pairs $(\class{C} \cap \class{W},\class{F})$ and $(\class{C},\class{W} \cap \class{F})$ are both hereditary (as defined in \ref{cotorsion-pairs}).

Gillespie extends in \cite[Thm.~3.3]{Gil2011} Hovey's correspondance, mentioned above, from the realm of abelian categories to the realm of weakly idempotent complete exact categories. More precisely, if $\cat{A}$ is just an exact category (not necessarily abelian), then an \emph{exact model structure} on $\cat{A}$ is a model structure on $\cat{A}$ which is compatible with the exact structure in the sense of \cite[Def.~3.1]{Gil2011}. If, in addition, $\cat{A}$ is weakly idempotent complete (\cite[Def.~2.2]{Gil2011}), then exact model structures on $\cat{A}$ correspond precisely to Hovey triples $(\class{C},\class{W},\class{F})$ in $\cat{A}$.
\end{ipg}

Recall from \cite[Cor.~1.2.7 and Thm.~1.2.10(i)]{modcat} that if $\cat{C}$ is any model category, then the inclusion $\cat{C}_\mathrm{cf} \to \cat{C}$ induces an equivalence \mbox{$\cat{C}_\mathrm{cf}/\mspace{-4mu}\sim \ \to \Ho(\cat{C})$}. Here $\cat{C}_\mathrm{cf}$ is the full subcategory of $\cat{C}$ whose objects are both cofibrant and fibrant, ``$\sim$'' is the (abstract) homotopy relation from \cite[Def.~1.2.4]{modcat}, and $\Ho(\cat{C})$ is the homotopy category of the model category $\cat{C}$ (that is, the localization of $\cat{C}$ with respect to the collection of weak equivalences).

\begin{ipg}
  \label{triangulated}
  Let $\cat{A}$ be a weakly idempotent complete exact category equipped with an exact model structure coming from a \textsl{hereditary} Hovery triple $(\class{C},\class{W},\class{F})$ in $\cat{A}$. As explained in \ref{Hovey-triple}, one has $\cat{A}_\mathrm{cf}=\class{C} \cap \class{F}$, which by \cite[Prop.~5.2(4)]{Gil2011}\,/\,\cite[Thm.~6.21(1)]{Stoviceksurvey} is a Frobenius category with $\class{C} \cap \class{W} \cap \class{F}$ as the class of projective-injective objects. By \cite[Prop.~4.4(5)]{Gil2011}\,/\,\cite[Lem. 6.16(3)]{Stoviceksurvey} two parallel morphisms in $\cat{A}_\mathrm{cf}=\class{C} \cap \class{F}$ are homotopic, in the (abstract) model categorical sense, if and only their difference factors through an object in \mbox{$\class{C} \cap \class{W} \cap \class{F}$}. Thus, \mbox{$\cat{A}_\mathrm{cf}/\mspace{-4mu}\sim$} is nothing but the stable category \smash{$\underline{\cat{A}_\mathrm{cf}}$} of the Frobenius category $\cat{A}_\mathrm{cf}$ (see the remarks pre\-ce\-ding Proposition~\ref{GProj-is-Frobenius}), so the category \mbox{$\cat{A}_\mathrm{cf}/\mspace{-4mu}\sim$} carries a natural triangulated structure. As mentioned above, one has an equivalence of categories \mbox{$\Ho(\cat{A}) \simeq \cat{A}_\mathrm{cf}/\mspace{-4mu}\sim$},~and~via this equivalence the homotopy category $\Ho(\cat{A})$ inherits a triangulated structure from \mbox{$\cat{A}_\mathrm{cf}/\mspace{-4mu}\sim$}. More precisely, the distinguished triangles in $\Ho(\cat{A})$ are, up to isomorphism, the images in $\Ho(\cat{A})$ of  distinguished triangles in $\underline{\cat{A}_\mathrm{cf}}=\cat{A}_\mathrm{cf}/\mspace{-4mu}\sim$ under the equivalence $\cat{A}_\mathrm{cf}/\mspace{-4mu}\sim \ \to \Ho(\cat{A})$. It is evident that when $\Ho(\cat{A})$ is equipped with this triangulated structure, then the equivalence \mbox{$\Ho(\cat{A}) \simeq \cat{A}_\mathrm{cf}/\mspace{-4mu}\sim$} (of ordinary categories \emph{a priori}) becomes an equivalence of triangulated categories, that is, the functors \mbox{$\Ho(\cat{A}) \leftrightarrows \cat{A}_\mathrm{cf}/\mspace{-4mu}\sim$}  are triangulated. 
\end{ipg}

\section{Sharp--Foxby adjunctions}

Recall from the beginning of Section~\ref{Preliminaries} that $\cat{A}$ always denotes any bicomplete abelian category with enough projectives and enough injectives. In this section, we give conditions on $\class A$ which ensure that $\sGProj(\class A)$ and $\sGInj(\class A)$ are equivalent as triangulated categories.

\begin{definition}
  \label{Up}
  Let $\Up[]$ be the full subcategory of $\cat{A}$ whose objects are given by
  \begin{displaymath}
    \Up[] = \left\{ M\in \class A \ | \ \GpdA(M)<\infty\right\}\,.
  \end{displaymath}
  Let and $\Cp[]$, $\Wp[]$, and $\Fp[]$ be the following subclasses of $\Up[]$:
  \begin{displaymath}
    \Cp[] = \GProj(\class A)\,, \quad
    \Wp[] = \left\{ M\in \class A \ | \ \pdA(M)<\infty\right\}, 
    \quad \text{and} \quad 
    \Fp[] = \Up[]\,.
  \end{displaymath}
  The classes $\Up[]$, $\Cp[]$, $\Wp[]$, and $\Fp[]$ depend on $\cat{A}$, and if necessay we use the more detailed notation $\Up$, $\Cp$, $\Wp$, and $\Fp$ instead. (The superscript ``$\symbolp$'' is supposed to give the reader associations to the word ``projective''.)
\end{definition}

\begin{definition}
  \label{Ui}
  Let $\Ui[]$ be the full subcategory of $\cat{A}$ whose objects are given by
  \begin{displaymath}
    \Ui[] = \left\{ N\in \class A \ | \ \GidA(N)<\infty\right\}\,.
  \end{displaymath}
  Let and $\Ci[]$, $\Wi[]$, and $\Fi[]$ be the following subclasses of $\Ui[]$:
  \begin{displaymath}
    \Ci[] = \Ui[]\,, \quad 
    \Wi[] = \left\{ N\in \class A \ | \ \idA(N)<\infty\right\}, 
    \quad \text{and} \quad
    \Fi[] = \GInj(\class A)\,.
  \end{displaymath}
    The classes $\Ui[]$, $\Ci[]$, $\Wi[]$, and $\Fi[]$ depend on $\cat{A}$, and if necessay we use the more detailed notation $\Ui$, $\Ci$, $\Wi$, and $\Fi$ instead. (The superscript ``$\symboli$'' is supposed to give the reader associations to the word ``injective''.)
\end{definition}

%\begin{remark}
%  By applying Definition \ref{Up} to the opposite category $\cat{A}^\mathrm{op}$, one obtains categories
%   $\Up[\!\cat{A}^\mathrm{op}]$, $\Cp[\!\cat{A}^\mathrm{op}]$, $\Wp[\!\cat{A}^\mathrm{op}]$, and $\Fp[\!\!\!\cat{A}^\mathrm{op}]$ which are equal to $\Ui$, $\Fi$, $\Wi$, and $\Ci$ (in that order!) as (Gorenstein) projective objects in $\cat{A}^\mathrm{op}$ correspond to (Gorenstein) injective objects in $\cat{A}$. In this sense, Definition~\ref{Ui} is superfluous, however, it is convenient to keep it as it is.
%\end{remark}   

\begin{lemma}
  \label{idempotent-complete-exact}
  The categories $\Up[]$ and $\Ui[]$ are additive and extension-closed subcategories of the abelian category $\cat{A}$; hence they are exact categories. Furthermore, $\Up[]$ and $\Ui[]$ are closed under direct summands in $\cat{A}$; hence they are idempotent complete.  
\end{lemma}

\begin{proof}
  In the case where $\cat{A}=\mathrm{Mod}(A)$ for a ring $A$, the assertions follow from 
  \cite[Prop. 2.19 and Thm.~2.24]{HHl04a} (and the dual statements about Gorenstein injective modules). By~in\-spec\-tion, one verifies that the same proofs work in any bicomplete abelian category $\cat{A}$ with enough projectives and enough injectives.
\end{proof}

We show in Theorems~\ref{thmUp} and \ref{thmUi} that $(\Cp[],\Wp[],\Fp[])$ and $(\Ci[],\Wi[],\Fi[])$ are Hovey triples (see \ref{Hovey-triple}) in the idempotent complete exact categories $\Up[]$ and $\Ui[]$.

\begin{definition}
  \label{Sharp-Foxby}
A \emph{Sharp--Foxby adjunction} on $\cat{A}$ is an adjunction $(S,T)$ of endofunctors on $\class A$ for which the following properties hold:
\begin{itemize}
\item[(SF1)] $S$ maps $\Up[]$ to $\Ui[]$ and it maps $\Wp[]$ to $\Wi[]$. 

\item[(SF2)] The restriction of $S$ to $\Up[]$ is exact: if $0 \to X' \to X \to X'' \to 0$ is an exact sequence in $\cat{A}$ with $X',X,X'' \in \Up[]$, then the sequence $0 \to SX' \to SX \to SX'' \to 0$ is exact.

\item[(SF3)] $T$ maps $\Ui[]$ to $\Up[]$ and it maps $\Wi[]$ to $\Wp[]$. 

\item[(SF4)] The restriction of $T$ to $\Ui[]$ is exact: if $0 \to Y' \to Y \to Y'' \to 0$ is an exact sequence in $\cat{A}$ with $Y',Y,Y'' \in \Ui[]$, then the sequence $0 \to TY' \to TY \to TY'' \to 0$ is exact.

\item[(SF5)] The unit of adjunction $\eta_X \colon X \to TSX$ is an isomorphism for every $X \in \Up[]$.

\item[(SF6)] The counit of adjunction $\varepsilon_Y \colon STY \to Y$ is an isomorphism for every $Y \in \Ui[]$.
\end{itemize}
\end{definition}

\begin{remark}
\label{Sharp-Foxby-remarks}
By (SF1), (SF3), (SF5), and (SF6) a Sharp--Foxby adjunction \mbox{$S \!: \cat{A} \rightleftarrows \cat{A} : T$} restricts to adjoint equivalences of categories $\Up[] \rightleftarrows \Ui[]$ and $\Wp[] \rightleftarrows \Wi[]$. By Lemma~\ref{idempotent-complete-exact} the categories $\Up[]$ and $\Ui[]$ have natural exact structures. Conditions (SF2) and (SF4) imply that the induced adjoint equivalence $\Up[] \rightleftarrows \Ui[]$ preserves the exact structure, i.e.~the functors are exact; thus it is an adjoint equivalence of exact categories.\footnote{\,If $\cat{E}$ and $\cat{E}'$ are exact categories and \mbox{$F \!: \cat{E} \rightleftarrows \cat{E}' : G$} is an adjoint equivalence of the underlying (ordinary) categories, then it does not automatically follow that the functors $F$ and $G$ are exact. Indeed, if $\cat{E}$ and $\cat{E}'$ have the same underlying category and the exact structure on 
$\cat{E}$ is coarser than that on $\cat{E}'$ (that is, every sequence which is exact in $\cat{E}$ is also exact in $\cat{E}'$ --- for example, $\cat{E}$ could have the trivial exact structure, in which the only ``exact'' sequences are the split exact ones, whereas $\cat{E}'$ could have any exact structure), then the identity functors \mbox{$\cat{E} \rightleftarrows \cat{E}'$} constitute an adjoint equivalence of the underlying categories where only $\cat{E} \to \cat{E}'$ is exact (but $\cat{E} \leftarrow \cat{E}'$ is not).}
\end{remark}

\enlargethispage{1.6ex}

The following example explains the terminology in Definition~\ref{Sharp-Foxby}.

\begin{example}   
\label{Auslander-Bass}
  Let $A$ be a commutative noetherian local Cohen--Macaulay ring with a dualizing module $D$. Foxby considered in \cite[\S1]{HBF72} two classes $\mathbf{A}(A)$ and $\mathbf{B}(A)$ of $A$-modules\footnote{\,In the literature, the classes $\mathbf{A}(A)$ and $\mathbf{B}(A)$ are referred to  as \emph{Foxby classes}. Sometimes, $\mathbf{A}(A)$ is called the \emph{Auslander class} and $\mathbf{B}(A)$ is called the \emph{Bass class}. Foxby himself \cite{HBF72} used the symbols $\Phi_D$ and $\Psi_D$ for these classes, but in the paper \cite{EJX-96b} by Enochs, Jenda, and Xu they are denoted by $\mathcal{G}_0$ and $\mathcal{J}_0$. We have adopted the symbols $\mathbf{A}(A)$~and~$\mathbf{B}(A)$ from the joint work of Avramov and Foxby; see for example \cite[\S3]{LLAHBF97}.}:
  
  A module $M$ is in $\mathbf{A}(A)$ if and only if $\Tor_i^A(D,M)=0$ and $\Ext_A^i(D,D\otimes_AM)=0$ for all $i>0$ and the natural homomorphism $\eta_M \colon M \to \Hom_A(D,D\otimes_AM)$ is an isomorphism.
  
  A module $N$ is in $\mathbf{B}(A)$ if and only if $\Ext_A^i(D,N)=0$ and $\Tor_i^A(D,\Hom_A(D,N))=0$ for all $i>0$ and the natural homomorphism $\varepsilon_N \colon D \otimes_A \Hom_A(D,N) \to N$ is an isomorphism.  
  
Foxby \cite{HBF72} proved that the adjunction \mbox{$(D\otimes_A\!-\,,\Hom_A(D,-))$} on $\mathrm{Mod}(A)$ restricts~to~an adjoint equivalence $\mathbf{A}(A)\rightleftarrows\mathbf{B}(A)$ and further to an adjoint equivalence \smash{\mbox{$\Wp[\mathrm{Mod}(A)] \mspace{-3mu}\rightleftarrows\mspace{-2mu} \Wi[\mathrm{Mod}(A)]$}} (see Definitions~\ref{Up} and \ref{Ui}). The latter is an extension of a result \cite[Thm.~(2.9)]{RYS72}~by~Sharp, which asserts that  \mbox{$D\otimes_A-$} and $\Hom_A(D,-)$ restrict to an adjoint equivalence between the categories
of \textsl{finitely generated} $A$-modules with finite projective dimension and \textsl{finitely ge\-ne\-ra\-ted} $A$-modules with finite injective dimension. Note that it is evident from the definitions
that the restriction of \mbox{$D\otimes_A-$} to $\mathbf{A}(A)$ and of $\Hom_A(D,-)$ to $\mathbf{B}(A)$ are exact functors.

By Enochs, Jenda, and Xu \cite[Cor.~2.4 and 2.6]{EJX-96b} an $A$-module belongs to $\mathbf{A}(A)$, respectively, $\mathbf{B}(A)$, if and only if it has finite Gorenstein projective dimension, respectively, finite Gorenstein injective dimension. Thus, in the notation from \ref{Up} and \ref{Ui} we have:
\begin{displaymath}
  \mathbf{A}(A)=\,\Up[\mathrm{Mod}(A)]  \qquad \text{and} \qquad
  \mathbf{B}(A)=\,\Ui[\mathrm{Mod}(A)].
\end{displaymath}
Consequently, \mbox{$(S,T)=(D\otimes_A\!-\,,\Hom_A(D,-))$} is a Sharp--Foxby adjunction on $\mathrm{Mod}(A)$. In view of \cite[Thms. 4.1 and 4.4]{CFH-06} this remains to be true if $A$ is any two-sided noetherian ring with a dualizing module $D$, that is, a dualizing complex concentrated in degree zero.
\end{example}

\begin{theorem}
  \label{thmUp}
  Consider the idempotent complete exact category $\Up[]$ from Lemma~\ref{idempotent-complete-exact}. The triple $(\Cp[],\Wp[],\Fp[])$ from Definition~\ref{Up} is a hereditary Hovey triple in $\Up[]$ (see \ref{Hovey-triple}). In~par\-ticular, $\Up[]$ has an exact model structure for which:
\begin{itemize}
\item[--] The cofibrant objects in $\Up[]$ are the Gorenstein projective objects in $\cat{A}$.
\item[--] The trivial objects in $\Up[]$ are the objects in $\class A$ with finite projective dimension.
\item[--] All objects in $\Up[]$ are fibrant.
\end{itemize}  
The homotopy category of this model category is equivalent, as a triangulated category, to the stable category of Gorenstein projective objects in $\cat{A}$; in symbols:
\begin{displaymath}
  \mathrm{Ho}(\Up[]) \,\simeq\, \sGProj(\class A)\,.
\end{displaymath}
\end{theorem}

\begin{remark}
A number of fundamental properties of Gorenstein projective modules, i.e. Gorenstein projective objects in the category $\cat{A}=\mathrm{Mod}(A)$ where $A$ is a ring, are recorded in e.g.~\cite{CFH-06,HHl04a}. The results we need about Gorenstein projective objects in a general abelian category (still bicomplete with enough projectives and enough injectives) can be proved as it is done for modules. We leave it to the reader to inspect the relevant proofs.
\end{remark}

\begin{proof}[Proof of Theorem~\ref{thmUp}]
It is well-known that $\Wp[]$ is a thick subcategory of $\cat{A}$ (and hence also of $\Up[]$). By \cite[Prop.~2.27]{HHl04a} the intersection $\Cp[] \cap \Wp[]$ equals the class $\operatorname{Proj}\cat{A}$ of projective objects in $\cat{A}$. Thus the pair $(\Cp[] \cap \Wp[], \Fp[])$ is equal to $(\operatorname{Proj}\cat{A},\,\Up[])$, which we now argue is a complete hereditary cotorsion pair in $\Up[]$. As $\Ext_\cat{A}^{\geqslant 1}(P,A)=0$ for all $P\in \operatorname{Proj}\cat{A}$ and all $A\in \Up[]$ (even all $A \in \cat{A}$), we get that  $(\operatorname{Proj}\cat{A})^\perp = \Up[]$ (as the ``$\perp$'' is only calculated inside of $\Up[]$) and that $\operatorname{Proj}\cat{A} \subseteq {}^\perp\Up[]$. To show that
$\operatorname{Proj}\cat{A} \supseteq {}^\perp\Up[]$ let $M\in {}^\perp\Up[] \ (\subseteq \Up[])$. By assumption, $\cat{A}$ has enough projectives, and hence there exists a short exact sequence in $\cat{A}$, 
\begin{equation}\label{eno-proj}
0\longrightarrow A\longrightarrow P\longrightarrow M\longrightarrow 0\,,
\end{equation}
where $P$ is projective. As $M$ belongs to $\Up[]$, so does $A$ by \cite[Thm.~2.24]{HHl04a}. By assumption, $\Ext_\cat{A}^1(M,A)=0$, so (\ref{eno-proj}) splits and hence $M\in \operatorname{Proj}\cat{A}$. This shows that $(\operatorname{Proj}\cat{A},\,\Up[])$ is a hereditary cotorsion pair. For completeness of this cotorsion pair, the sequence (\ref{eno-proj}) shows that the pair has enough projectives. The trivial exact sequence $0\to M\to M\to 0\to 0$ (for any $M$ in $\Up[]$) shows that the pair has enough injectives.

Next we show that $(\Cp[],\Wp[]\cap \Fp[])=(\GProj\class A,\Wp[])$ is a complete hereditary cotorsion pair in $\Up[]$. By \cite[Thm.~2.20]{HHl04a} we have $\Ext_\cat{A}^{\geqslant 1}(G,A)=0$ for all $G\in \GProj\class A$ and $A\in \Wp[]$, and hence we get $\GProj\class A \subseteq {}^\perp\Wp[]$ and $(\GProj\class A)^\perp \supseteq \Wp[]$. To show that $\GProj\class A \supseteq {}^\perp\Wp[]$, let $M\in {}^\perp\Wp[] \ (\subseteq \Up[])$. By \cite[Thm~2.10]{HHl04a} there exists a short exact sequence
\begin{equation}\label{eno-proj2}
0\longrightarrow A\longrightarrow G\longrightarrow M\longrightarrow 0
\end{equation}
with $G\in \GProj\class A$ and $A \in \Wp[]$. By assumption, $\Ext_\cat{A}^1(M,A)=0$, so (\ref{eno-proj2}) splits and hence $M$ is a direct summand in $G$. By \cite[Thm~2.5]{HHl04a} (see also Prop.~1.4 in \emph{loc.~cit.}) the class $\GProj\cat{A}$ is closed under direct summands (here we use our assumption that $\cat{A}$ is cocomplete, or at least that $\cat{A}$ has countable coproducts), and it follows that $M$ itself belongs to $\GProj\cat{A}$.
To show $(\GProj\class A)^\perp \subseteq \Wp[]$, assume that $M\in (\GProj\class A)^\perp \ (\subseteq \Up[])$. By \cite[Lem.~2.17]{CFH-06} there is a short exact sequence
\begin{equation}\label{eno-inj2}
 0\longrightarrow M\longrightarrow A'\longrightarrow G'\longrightarrow 0
\end{equation} 
where $G'\in \GProj\class A$ and $\pd{A'}=\GpdA(M)<\infty$, that is, $A'$ is in $\Wp[]$. By assumption, $\Ext_\cat{A}^1(G',M)=0$, so (\ref{eno-inj2}) splits and hence $M$ also belongs to $\Wp[]$ (which is thick). Thus $(\GProj\class A,\Wp[])$ is a hereditary cotorsion pair in $\Up[]$, and the existence of the sequences (\ref{eno-proj2}) and (\ref{eno-inj2}) shows that this cotorsion pair is complete.

These arguments prove that $(\Cp[],\Wp[],\Fp[])$ is a hereditary Hovey triple in $\Up[]$. In view of the equalities $\Cp[] \cap \Fp[] = \GProj\class A$ and $\Cp[] \cap \Wp[] \cap \Fp[] = \operatorname{Proj}\cat{A}$, where the latter is by \cite[Prop~2.27]{HHl04a}, the rest of the theorem now follows from \ref{Hovey-triple} and \ref{triangulated} (and Proposition~\ref{GProj-is-Frobenius}).
\end{proof}

\begin{theorem}
  \label{thmUi}  
    Consider the idempotent complete exact category $\Ui[]$ from Lemma~\ref{idempotent-complete-exact}.~The triple $(\Ci[],\Wi[],\Fi[])$ from Definition~\ref{Ui} is a hereditary Hovey triple in $\Ui[]$ (see \ref{Hovey-triple}). In~par\-ticular, $\Ui[]$ has an exact model structure for which:
\begin{itemize}
\item[--] All objects in $\Ui[]$ are cofibrant.
\item[--] The trivial objects in $\Ui[]$ are the objects in $\class A$ with finite injective dimension.
\item[--] The fibrant objects in $\Ui[]$ are the Gorenstein injective objects in $\cat{A}$.
\end{itemize}  
The homotopy category of this model category is equivalent, as a triangulated category, to the stable category of Gorenstein injective objects in $\cat{A}$; in symbols:
\begin{displaymath}
  \mathrm{Ho}(\Ui[]) \,\simeq\, \sGInj(\class A)\,.
\end{displaymath}
\end{theorem}

\begin{proof}
  Dual to the proof of Theorem~\ref{Up}.
\end{proof}

\enlargethispage{1.2ex}

Our next goal is to show that a Sharp--Foxby adjunction on $\class A$ induces a Quillen equivalence between the model categories $\Up[]$ and $\Ui[]$. To this end, the next result will be useful.

\begin{proposition}
\label{Dalezios}
Let $\mathcal M$ and $\mathcal M'$ be two weakly idempotent complete exact model categories with associated Hovey triples $(\class C,\class W,\class F)$ and $(\class C',\class W',\class F')$; see \ref{Hovey-triple}. Assume that~$(F,G)$ is a Quillen adjunction $\class M \rightleftarrows \class M'$ where the functors $F$ and $G$ are exact and satisfy $F(\class W)\subseteq \class W'$ and $G(\class W') \subseteq \class W$. Then $(F,G)$
is a Quillen equivalence if and only if the unit $\eta_{X} \colon X\rightarrow GFX$ is a weak equivalence for every $X\in\class C$ and the counit $\varepsilon_{Y} \colon FGY\rightarrow Y$ is a weak equivalence for every $Y\in\class F'$.
\end{proposition}

\begin{proof}
  Write $Q$ for the cofibrant replacement functor in $\mathcal M$ and $q_X \colon QX \to X$ for the natural trivial fibration ($X \in \mathcal M$). Similarly, write $R$ for the fibrant replacement functor in $\mathcal M'$ and $r_Y \colon Y \to RY$ for the natural trivial cofibration ($Y \in \mathcal M'$). By \cite[Prop.~1.3.13]{modcat} we have that $(F,G)$ is a Quillen equivalence if and only if the composite
\begin{displaymath}
  \xymatrix{
    X \ar[r]^-{\eta_{X}} & GFX \ar[r]^-{Gr_{FX}} & GRFX
  }
\end{displaymath}  
is a weak equivalence for all $X\in\class C$ and the composite
\begin{displaymath}
  \xymatrix{
    FQGY \ar[r]^-{Fq_{GY}} & FGY \ar[r]^-{\varepsilon_Y} & Y  
  }
\end{displaymath}
is a weak equivalence for all $Y\in\class F'$. We claim that the morphisms $Gr_{FX}$ and $Fq_{GY}$ are always weak equivalences for every $X\in\class M$ and $Y\in\class M'$ (which proves the assertion by the 2-out-of-3 property for weak equivalences). We only show that $Gr_{FX}$ is a weak equivalence. The fact that $r_{FX} \colon FX \to RFX$ is a trivial cofibration means, by definition~\cite[~Def.~3.1]{Gil2011} of an exact model structure, that $r_{FX}$ is an
admissible monomorphism with a trivially cofibrant cokernel, that is, one has a conflation (a short exact sequence)
\begin{displaymath}
  \xymatrix{
    FX \ \ar@{>->}[r]_-{\simeq}^-{r_{FX}} & RFX \ar@{->>}[r]^-{\pi} & C
  }
\end{displaymath}
in $\mathcal M'$ where $C$ is trivially cofibrant, that is, $C \in \class C' \cap \class W'$ (and 
$RFX$ is of course fibrant). By applying the exact functor $G$ to the sequence above, we get a conflation in $\mathcal M$, which is the bottom row of the following pullback diagram:
\begin{displaymath}
  \xymatrix{
    GFX \ \ar@{>->}[r]^-{\iota}_-{\simeq} \ar@{=}[d] & 
    T \ar@{->>}[r]^-{\varrho} \ar@{->>}[d]^-{\varphi}_-{\simeq} & 
    QGC \ar@{->>}[d]_-{\simeq}^-{q_{GC}} \\
    GFX \ \ar@{>->}[r]^-{Gr_{FX}} & GRFX \ar@{->>}[r]^-{G\pi} & GC\,.  
  }
\end{displaymath}
Note that this pullback diagram really exists; indeed, by definition of an exact category, any pullback of an admissible epimorphism exists and admissible epimorphisms are stable under pullbacks. In particular, $\varrho$ is an admissible epimorphism (and $\varrho$ has the same kernel as $G\pi$;  cf.~Freyd~\cite[Thm.~2.52]{abcat}). Since $C \in \class W'$ we have $GC\in\class W$ by assumption. Since one always has $Q(\class W) \subseteq \class W$, it follows that $QGC \in \class W$, and hence $QGC \in \class C \cap \class W$ (as $QY$ is always cofibrant). This means that $\iota$ is a trivial cofibration. In any model category, the class of trivial fibrations is stable under pullbacks by \cite[Cor.~1.1.11]{hovey}; thus the fact that $q_{GC}$ is a trivial fibration forces $\varphi$ to be the same. As $\iota$ and $\varphi$ are, in particular, weak equivalences, so is their composite $Gr_{FX}=\varphi\circ \iota$, as desired.
\end{proof}

\enlargethispage{3.1ex}

\begin{theorem}
  \label{Ho}
 A Sharp--Foxby adjunction $(S,T)$ on $\class A$ induces a Quillen equivalence between the model categories $\Up[]$ and $\Ui[]$ constructed in Theorems~\ref{thmUp} and \ref{thmUi}. Thus the total (left/right) derived functors of $S$ and $T$ yield an adjoint equivalence of the corresponding homotopy categories,
\begin{equation}
  \label{eq:derived-functors}
  \xymatrix@C=3pc{
     \mathrm{Ho}(\Up[]) \ar@<0.7ex>[r]^-{\mathbf{L}S} & 
     \mathrm{Ho}(\Ui[]) \ar@<0.7ex>[l]^-{\mathbf{R}T}
  }\!.
\end{equation} 
In fact, this is an equivalence of triangulated categories.
\end{theorem}

\begin{proof}
  As mentioned Remark~\ref{Sharp-Foxby-remarks}, a Sharp--Foxby adjunction $(S,T)$ on $\class A$ induces an exact adjoint equivalence between $\Up[]$ and $\Ui[]$ with $S(\Wp[]) \subseteq \Wi[]$ and $T(\Wi[]) \subseteq \Wp[]$. Hence the unit $\eta_{X} \colon X\rightarrow TSX$ is an isomorphism, and hence also a weak equivalence, for all $X \in \Up[]$ (in particular for $X\in\Cp[]$); and the counit $\varepsilon_{Y} \colon STY\rightarrow Y$ is an isomorphism, and hence also a weak equivalence, for all $Y \in \Ui[]$ (in particular for $Y\in\Fi[]$). Thus, if we can show that $(S,T)$ is a Quillen adjunction $\Up[] \rightleftarrows \Ui[]$, then Proposition~\ref{Dalezios} will imply that it is in fact a Quillen equivalence (as claimed). To show this, it must be argued that $S \colon \Up[] \to \Ui[]$ is a left Quillen functor (see \cite[Def.~1.3.1]{modcat}), that is, we must argue that $S$ maps (trivial) cofibrations in $\Up[]$ to (trivial) cofibrations in $\Ui[]$. Let $f$ be a (trivial) cofibration in $\Up[]$, that is, $f$ is an admissible monomorphism with a (trivially) cofibrant cokernel $C$ (see \cite[~Def.~3.1]{Gil2011}). Since $S$ is exact, it follows that $S\!f$ is an admissible monomorphism in $\Ui[]$ with cokernel $S\!C$. Hence, we only need to prove that $S$ maps (trivially) cofibrant objects in $\Up[]$ to (trivially) cofibrant objects in $\Ui[]$. However, this is clear as every object in $\Ui[]$ is cofibrant, see Theorem~\ref{thmUi}, and since we have $S(\Wp[]) \subseteq \Wi[]$.
  
  Having established that $(S,T)$ yields a Quillen equivalence $\Up[] \rightleftarrows \Ui[]$, the adjoint equivalence of homotopy categories displayed in (\ref{eq:derived-functors}) follows from \cite[Prop.~1.3.13]{modcat}.
  
  It remains to see that the functors $\mathbf{L}S$ and $\mathbf{R}T$ are triangulated. By \cite[Lem.~5.3.6]{Nee} it suffices to prove that $\mathbf{L}S$ is triangulated, because then its right adjoint $\mathbf{R}T$ will automatically be triangulated as well. Recall from \ref{triangulated} that the distinguished triangles in $\Ho(\Up[])$ are, up to isomorphism, the images in $\Ho(\Up[])$ of distinguished triangles in $\sGProj(\class A)$ under the equivalence $\sGProj(\class A) \to \Ho(\Up[])$ (see also Theorem \ref{thmUp}). 
  
  At this point we need to recall from \cite[Chap.~I\S2.5]{HapT} how the triangulated structure on the stable category $\sGProj(\class A)$ is defined. For every morphism $u \colon G \to G'$ in the Frobenius category $\GProj(\class A)$ choose a short exact sequence (a conflation) \smash{\mbox{$G \stackrel{\text{\raisebox{2.5pt}{$i$}}}{\smash{\rightarrowtail}} P \stackrel{\text{\raisebox{3.5pt}{$p$}}}{\smash{\twoheadrightarrow}} \tilde{G}$}} in $\GProj(\class A)$ where $P$ is a projective-injective object, that is, $P \in \operatorname{Proj}(\class A)$. The object $\tilde{G}$ is the suspension of $G$; in symbols, $\tilde{G}=\upSigma G$ (the assignment $G \mapsto \tilde{G}=\upSigma G$ is not functorial on $\GProj(\class A)$, but it is functorial on $\sGProj(\class A)$). Then consider the pushout diagram in $\GProj(\class A)$,  
\begin{equation}
  \label{eq:2x3}
  \begin{gathered}
  \xymatrix{
    G \ \ar[d]_-{u} \ar@{>->}[r]^{i} & P \ar[d]^-{t} \ar@{->>}[r]^{p} & \tilde{G} \ar@{=}[d] \phantom{\; .}
    \\
    G' \ar@{}[ur]|-{\mathrm{pushout}} \ar@{>->}[r]^-{v} & G'' \ar@{->>}[r]^-{w} & \tilde{G} \;.  
  }
  \end{gathered}  
\end{equation}
The diagram
\begin{equation}
  \label{eq:standard-triangle}
  \xymatrix{
    G \ar[r]^-{u} & G' \ar[r]^-{v} & G'' \ar[r]^-{w} & \tilde{G}
  }\!,
\end{equation}
considered as a diagram in $\sGProj(\class A)$, is called a \emph{standard triangle}. By definition, a distinguished triangle in $\sGProj(\class A)$ is a diagram in this category which is isomorphic to some standard triangle. The triangulated structure on $\sGInj(\class A)$ is defined similarly.

We must show that the functor $\mathbf{L}S$ maps every distinguished triangle $\upDelta$ in $\Ho(\Up[])$ to a distinguished triangle in $\Ho(\Ui[])$. By the considerations above, we may assume that $\upDelta$ is the image in $\Ho(\Ui[])$ of a standard triangle  (\ref{eq:standard-triangle}) in $\sGProj(\class A)$. By definition, see \cite[~Def.~1.3.6]{modcat}, the action of the functor $\mathbf{L}S$ on an object $X$ in $\Ho(\Up[])$ is $\mathbf{L}S\!(X) = S\!QX$ where $QX$ is a cofibrant replacement of $X$. As the objects in (\ref{eq:standard-triangle}) are already cofibrant~in~$\Up[]$, see Theorem~\ref{thmUp}, the diagram $\mathbf{L}S\!(\upDelta)$ is nothing but
\begin{equation}
  \label{eq:S-standard-triangle}
  \xymatrix{
    S\!G \ar[r]^-{S\!u} & S\!G' \ar[r]^-{S\!v} & S\!G'' \ar[r]^-{S\!w} & S\!\tilde{G}
  }\!,
\end{equation}
which we must show is a distinguished triangle in $\Ho(\Ui[])$. Since the pair $(\Ci[] \cap \Wi[], \Fi[]) = (\Wi[],\GInj \class A)$ is a hereditary cotorsion pair in $\Ui[]$, see Theorem~\ref{thmUi} and Definition~\ref{Ui}, it follows from \cite[Lem.~6.20]{Stoviceksurvey} that we can find a diagram in $\Ui[]$,
\begin{equation}
  \label{eq:3x3}
  \begin{gathered}
  \xymatrix@R=1.5pc{
    S\!G \, \ar@{>->}[d]^-{h} \ar@{>->}[r]^-{S\!i} & 
    S\!P \ar@{>->}[d]^-{e} \ar@{->>}[r]^-{S\!p} & 
    S\!\tilde{G} \ar@{>->}[d]^-{\tilde{h}}
    \\
    H \, \ar@{->>}[d] \ar@{>->}[r]^-{i_0} & 
    E \ar@{->>}[d] \ar@{->>}[r]^-{p_0} & 
    \tilde{H} \ar@{->>}[d]
    \\
    J \, \ar@{>->}[r] & I \ar@{->>}[r] & \tilde{J}     
  }
  \end{gathered}  
\end{equation}
whose rows and columns are conflations, where $H,E,\tilde{H}$ are Gorenstein injective, and where $J,I,\tilde{J}$ have finite injective dimension. As $P \in \operatorname{Proj}\cat{A} \subseteq \Wp[]$ we have $S\!P \in \Wi[]$, that is, $S\!P$ has finite injective dimension. It follows from the middle column in (\ref{eq:3x3}) that $E$ has finite injective dimension, and since $E$ is also Gorenstein injective it must be injective (this is immediate from the  definition, \ref{Gorenstein-objects}, of Gorenstein injective objects). Let \smash{\mbox{$S\!G' \stackrel{\text{\raisebox{2.5pt}{$h'$}}}{\smash{\rightarrowtail}} H' \twoheadrightarrow J'$}} be a short exact sequence with $H' \in \GInj \class A$ and  $J' \in \Wi[]$. The morphism $h \colon S\!G \to H$ is a (special) Gorenstein injective preenvelope of $S\!G$ since it is monic and its cokernel $J \in \Wi[]$ satisfies $\Ext_{\cat{A}}^1(J,X)=0$ for all $X \in  \GInj(\class A)$; see \cite[Prop.~2.1.4]{Xu}. Thus, the morphism $h'S\!u \colon S\!G \to H' \in \GInj(\class A)$ lifts to a morphism $u_0 \colon H \to H'$ such that $u_0h = h'S\!u$. This gives commutativity of the left wall in the following diagram:
    \begin{equation}
    \label{eq:big}
    \begin{gathered}
      \xymatrix@!=0.5pc{ {} & H \, \ar'[d][dd]_-{u_0}
          \ar@{>->}[rr]^-{i_0} & & E \ar'[d][dd]_-{t_0} 
          \ar@{->>}[rr]^-{p_0} 
          & & \tilde{H} \ar@{=}[dd]
          \\
          S\!G \,
          \ar@{>->}[ur]^(0.45){h}
          \ar[dd]_-{S\!u}
          \ar@{>->}[rr]^(0.70){S\!i} & & 
          S\!P
          \ar@{>->}[ur]^(0.45){e}
          \ar[dd]_(0.70){S\!t}
          \ar@{->>}[rr]^(0.70){S\!p} & & S\!\tilde{G}
          \ar@{>->}[ur]^(0.45){\tilde{h}}
          \ar@{=}[dd]
          \\
          {} & H' \, \ar@{>->}'[r]^-{v_0}[rr] & & 
          H'' \ar@{->>}'[r]^-{w_0}[rr] 
          & & \tilde{H}
          \\
          S\!G' \,
          \ar@{>->}[ur]_(0.6){h'}
          \ar@{>->}[rr]_-{S\!v} & & S\!G''
          \ar@{>..>}[ur]_(0.6){h''}
          \ar@{->>}[rr]_-{S\!w} & & S\!\tilde{G}
          \ar@{>->}[ur]_(0.6){\tilde{h}}
        }
    \end{gathered}
    \end{equation}
The top wall in (\ref{eq:big}) is just the upper half of the commutative diagram (\ref{eq:3x3}). The back wall is the (commutative) pushout diagram of the morphisms \smash{\mbox{$H'\! \stackrel{\text{\raisebox{3.5pt}{$u_0\mspace{-7mu}$}}}{\smash{\leftarrow}} H \stackrel{\text{\raisebox{3.5pt}{$i_0$}}}{\smash{\rightarrowtail}} E$}}. The right wall is evidently commutative. The front wall in (\ref{eq:big}) is obtained by applying the exact functor $S$ to the diagram (\ref{eq:2x3}). Since $S$ is a left adjoint functor, it preserves colimits, so the front wall in (\ref{eq:big}) is (still) a pushout diagram. As $(v_0h')S\!u = v_0u_0h = t_0i_0h = (t_0e)S\!i$ and since $S\!G''$ is the pushout of
\smash{\mbox{$S\!G'\! \stackrel{\text{\raisebox{2.3pt}{$S\!u\mspace{-7mu}$}}}{\smash{\longleftarrow}} S\!G \stackrel{\text{\raisebox{2.3pt}{$S\!i$}}}{\smash{\longrightarrow}} S\!P$}}, there exists a (unique) morphism $h'' \colon S\!G'' \to H''$ such that $h''S\!v = v_0h'$ and $h''S\!t = t_0e$. The first of these identities show that the left square in the bottom wall in (\ref{eq:big}) is commutative. It follows from the universal property of the pushout $S\!G''$ that the right square in the bottom wall is commutative as well.
%To see that also the right square in the bottom wall in is commutative, we argue as follows. As $p_0eS\!i = \tilde{h}\mspace{1mu}S\!\mspace{-1mu}pS\!i = \tilde{h}\mspace{1mu}0 = 0S\!u$ there is by the universal property of the pushout a unique morphism $\varphi$ that makes the next diagram commutative:
%\begin{equation}
%\label{eq:pushout0}
%\begin{gathered}
%\xymatrix@R=1.5pc@C=1.5pc{
%  S\!G \ar[d]_-{S\!u} \ar[r]^-{S\!i} & S\!P \ar[d]^-{S\!t} \ar@/^1pc/[ddr]^-{p_0e} & {}
%  \\
%  S\!G' \ar@/_1pc/[drr]_-{0} \ar[r]_-{S\!v} & S\!G'' \ar[dr]^-{\varphi} & {}  
%  \\
%  {} & {} & \tilde{H}\,.
%  }
%\end{gathered}  
%\end{equation}
%We have $w_0h''S\!v = w_0v_0h' = 0h' = 0$ and $w_0h''S\!t = w_0t_0e = p_0e$, and also $\tilde{h}\mspace{1mu}S\!wS\!v = \tilde{h}\mspace{1mu}0 = 0$ and $\tilde{h}\mspace{1mu}S\!wS\!t = \tilde{h}\mspace{1mu}S\!\mspace{-1mu}p = p_0e$. Thus, both $\varphi = w_0h''$ and $\varphi = \tilde{h}\mspace{1mu}S\!w$ make the diagram (\ref{eq:pushout0})~com\-mu\-tative, so we must have $w_0h''=\tilde{h}\mspace{1mu}S\!w$. Hence the entire bottom wall in (\ref{eq:big}) is commutative. 
By applying the Snake Lemma to this bottom wall, we see that $h''$ is monic (as $h'$ and $\tilde{h}$ are so) and that the cokernel $J''$ of $h''$ sits in a short exact sequence $0 \to J' \to J'' \to \tilde{J} \to 0$. Since $J',\tilde{J} \in \Wi[]$ it follows that  $J'' \in \Wi[]$. Since $h$, $h'$, $h''$, and $\tilde{h}$ are (admissible) monomorphisms in $\Ui[]$ whose cokernels belong to $\Wi[]$ (which are the trivally cofibrant objects in $\Ui[]$), they are trivial cofibrations in the exact model structure on $\Ui[]$; see \cite[~Def.~3.1]{Gil2011}. In particular, $h$, $h'$, $h''$, and $\tilde{h}$ are weak equivalences in $\Ui[]$ and therefore isomorphisms in $\Ho(\Ui[])$.  The commutative diagram (\ref{eq:big}) now shows that in the homotopy category $\Ho(\Ui[])$, the diagram (\ref{eq:S-standard-triangle}) is isomorphic to 
\begin{equation}
  \label{eq:standard-triangle-GInj}
  \xymatrix{
    H \ar[r]^-{u_0} & H' \ar[r]^-{v_0} & H'' \ar[r]^-{w_0} & \tilde{H}
  }\!.
\end{equation}
By definition, and by commutativity of the back wall in (\ref{eq:big}), the diagram   (\ref{eq:standard-triangle-GInj}) is a standard triangle in $\sGInj(\class A)$, and consequently, (\ref{eq:S-standard-triangle}) is a distinguished triangle in $\Ho(\Ui[])$.
\end{proof}

\begin{corollary}
  \label{cor:SF-imples-GProj-GInj}
  If there exists a Sharp--Foxby adjunction $(S,T)$ on $\class A$, then there is an equivalence of triangulated categories, $\sGProj(\class A) \simeq \sGInj(\class A)$.
\end{corollary}

\begin{proof}
By Theorems~\ref{thmUp}, \ref{Ho}, and \ref{thmUi} there are the following equivalences of triangulated categories, $\sGProj(\class A) \,\simeq\, \mathrm{Ho}(\Up[]) \,\simeq\, \mathrm{Ho}(\Ui[]) \,\simeq\, \sGInj(\class A)$.
\end{proof}

\begin{remark}
\label{virtally-gorenstein}
Before closing this section, we record a biproduct of Proposition~\ref{Dalezios} concerning virtually Gorenstein rings, which should be well known. We recall from \cite{ABgHKr,Beligiannis-Reiten-ho-aspects} that an Artin algebra $A$ is called virtually Gorenstein if $(\GProj(A))^{\bot}=^{\bot}\!\!(\GInj(A))$. The same notion for commutative rings has also been studied in \cite{Khoshchehreh-gorenstein-homology}. In what follows, assume that $A$ is an Artin algebra or a commutative noetherian ring with finite Krull dimension. In both cases, it is well known \cite{Beligiannis-Reiten-ho-aspects,gillespie-recollement,HKr05} that there are Hovey triples 
\[(\GProj(A),(\GProj(A))^{\bot},\mathrm{Mod}(A))\,\,\,\,\,\,\ \mbox{and} \,\,\,\,\,\,\  (\mathrm{Mod}(A),^{\bot}\!\!(\GInj(A)),\GInj(A)).\]
Applying Proposition~\ref{Dalezios} in the case where $F=G=I_{\mathrm{Mod}(A)}$, we obtain that virtually Gorensteiness of $A$ implies that the identity is a Quillen equivalence between the two model structures. Therefore the homotopy categories of these two models are, in fact, isomorphic. In case $A$ is, in addition, commutative Gorenstein we recover the analogous statement for Gorenstein rings (see the comments after Theorem 8.6 in \cite{hovey}).
\end{remark}

\enlargethispage{4.3ex}

\section{The case of chain complexes}
\label{sec:Ch}

Recall from the beginning of Section~\ref{Preliminaries} that $\cat{A}$ always denotes any bicomplete abelian category with enough projectives and enough injectives. In this section, we consider the abelian category $\Ch(\mathcal{A})$ of unbounded chain complexes in $\class A$ and prove that, under suitable conditions, a Sharp--Foxby adjunction $(S,T)$ on $\class A$ induces a Sharp--Foxby adjunction on $\Ch(\mathcal{A})$ by degreewise application of the functors $S$ and $T$. First we recall the following.
 
%First we recall the finitistic dimensions of an abelian category.
\begin{ipg} 
   The \emph{finitistic projective dimension}, $\FPD(\class A)$, of $\class A$ is defined as 
$$\FPD(\class A)=\mathrm{sup}\{\mathrm{pd}_{\class A} M \,|\, \textrm{ $M$ is an object in $\class A$ with finite projective dimension}\}.$$ 
Dually, the \emph{finitistic injective dimension}, $\FID(\class A)$, of $\class A$ is
$$\FID(\class A)=\mathrm{sup}\{\mathrm{id}_{\class A} M \,|\, \textrm{ $M$ is an object in $\class A$ with finite injective dimension}\}.$$ 
The \emph{finitistic Gorenstein projective dimension}, $\FGPD(\class A)$, and the \emph{finitistic Gorenstein injective dimension},  $\FGID(\class A)$, are defined similarly.
%as follows:
%\begin{align*}
%  \FGPD(\class A) &=
%  \mathrm{sup}\left\{\mathrm{Gpd}_{\class A} M \,\left|\!\!\!\!
%  \begin{array}{c}
%    \textrm{ $M$ is an object in $\class A$ with finite} \\ 
%    \textrm{ Gorenstein projective dimension}
%  \end{array} 
%  \right. \!\!\!\right\} \quad \text{and}
%  \\
%  \FGID(\class A) &=
%  \mathrm{sup}\left\{\mathrm{Gid}_{\class A} M \,\ \left|\!\!\!\!
%  \begin{array}{c}
%    \textrm{ $M$ is an object in $\class A$ with finite} \\ 
%    \textrm{ Gorenstein injective dimension}
%  \end{array}   
%  \right. \!\!\!\right\}.
%\end{align*}
\end{ipg}

For most abelian categories that appear in applications, the finitistic dimensions defined above turn out to be finite. As in \cite[(proofs of) Thms.~2.28 and 2.29]{HHl04a} one easily proves:

\begin{lemma}
  \label{finitistic}
  There are equalities $\FGPD(\class A)=\FPD(\class A)$ and $\FGID(\class A)=\FID(\class A)$. Thus, if $\,\FPD(\class A)$, respectively, $\FID(\class A)$, is finite, then so is $\FGPD(\class A)$, respectively, $\FGID(\class A)$. \qed
\end{lemma}

In $\mathcal{A}$ we have the subcategories $\Up$, $\Cp$, $\Wp$ and $\Fp$ from Definition \ref{Up}. Similarly, in $\cat{B}=\Ch(\mathcal{A})$ we have the subcategories $\Up[\cat{B}]$, $\Cp[\cat{B}]$, $\Wp[\cat{B}]$ and $\Fp[\cat{B}]$. The following result explains the relation between all these subcategories.

\begin{proposition} 
\label{Ch-p}
Assume that $\FPD(\class A)<\infty$ and let $X=\cdots \rightarrow X_{n+1}\rightarrow X_{n}\rightarrow X_{n-1}\rightarrow\cdots$ be an object in $\mathcal{B}:=\Ch(\class A)$. The following conclusions hold.
\begin{itemize}
\item [(i) ] $X$ belongs to $\Up[\cat{B}]$ if and only if every $X_{n}$ belongs to $\Up$.
\item [(ii) ] $X$ belongs to $\Cp[\cat{B}]$ if and only if every $X_{n}$ belongs to $\Cp$.
\item [(iii) ] $X$ belongs to $\Wp[\cat{B}]$ if and only if $X$ is exact and every cycle $\mathrm{Z}_{n}(X)$ belongs to $\Wp$.
\item [(iv) ] $X$ belongs to $\Fp[\cat{B}]$ if and only if every $X_{n}$ belongs to $\Fp$.
\end{itemize}
\end{proposition}

\begin{proof}
Part (ii) is proved in \cite[Thm.~2.2]{yang-liu-Gorenstein-complexes} in the case $\mathcal{A}=\mathrm{Mod}(A)$ where $A$ is any ring, but the proof works in any abelian category (with enough projectives). 

In view of (ii), the ``only if'' part in (i) is clear.
To prove the ``if'' part in (i), assume that every $X_{n}$ is in $\Up$, that is, $\GpdA(X_n)<\infty$. By our assumption $\FPD(\class A)<\infty$ and by Lemma~\ref{finitistic}, it follows that $s=\sup\{\GpdA(X_{n}) \,|\, n\in\mathbb{Z}\}$ belongs to $\mathbb{N}_0$. The proof is now by induction on $s$. If $s=0$, then $X$ is even in $\Cp[\cat{B}] \subseteq \Up[\cat{B}]$ by part (ii). Now assume that $s>0$. Choose any exact sequence
\begin{displaymath}
  0 \longrightarrow K \longrightarrow P^{s-1} \longrightarrow \,\cdots\, \longrightarrow P^1 \longrightarrow P^0 \longrightarrow X \longrightarrow 0
\end{displaymath}
in $\mathcal{B}=\Ch(\class A)$ where $P^0,\ldots,P^{s-1}$ are complexes consisting of projective objects in $\cat{A}$. For each $n \in \mathbb{Z}$ we have an exact sequence $0 \to K_n \to P^{s-1}_n \to P^1_n \to P^0_n \to X_n \to 0$ in $\cat{A}$, and since $P^0_n,\ldots,P^{s-1}_n$ are projectives and $\GpdA(X_{n}) \leqslant s$, it follows that $K_n$ is Gorenstein projective; cf.~\cite[(proof of) Prop.~2.7]{HHl04a}. Thus, $K$ is a complex of Gorenstein projective objects in $\cat{A}$, which by (ii) means that $K$ is a Gorenstein projective object in $\mathcal{B}=\Ch(\class A)$. So  the exact sequence displayed above shows that $\mathrm{Gpd}_{\cat{B}}(X) \leqslant s < \infty$, that is, $X \in \Up[\cat{B}]$.

To prove (iii), let $X\in \Wp[\cat{B}]$, which means that we have an exact sequence 
\begin{equation}
  \label{m}
  0 \longrightarrow P^m \longrightarrow \,\cdots\, \longrightarrow P^1 \longrightarrow P^0 \longrightarrow X \longrightarrow 0
\end{equation}
in $\mathcal{B}=\Ch(\class A)$ where $P^0,\ldots,P^m$ are projective objects; i.e.~each $P^i$ is a split exact complex of projective objects in $\cat{A}$, and thus each cycle $\mathrm{Z}_n(P^i)$ is also projective~in~$\cat{A}$. As the complexes $P^0,\ldots,P^m$ are, in particular, exact, so is $X$ (and the same are all the kernel and cokernel complexes of the chain maps that appear in (\ref{m})). This implies that the functor $\mathrm{Z}_n(-)$ leaves the sequence (\ref{m}) exact, and the hereby obtained exact sequence
\begin{displaymath}
 % \label{m}
  0 \longrightarrow \mathrm{Z}_n(P^m) \longrightarrow \,\cdots\, \longrightarrow \mathrm{Z}_n(P^1) \longrightarrow \mathrm{Z}_n(P^0) \longrightarrow \mathrm{Z}_n(X) \longrightarrow 0
\end{displaymath}
shows that $\mathrm{Z}_n(X)$ has finite projective dimension in $\cat{A}$, that is, 
$\mathrm{Z}_{n}(X)$ belongs to $\Wp$.

The proof of the ``if'' part in (iii) is based on a standard construction; see (the dual of) \cite[Thm.~3.1.3]{GR99} (for this argument to work we make use the hypothesis $\FPD(\class A)<\infty$).

Part (iv) is just a repetition of part (i) since $\Fp[\cat{B}]=\Up[\cat{B}]$ and $\Fp = \Up$.
\end{proof}

\enlargethispage{0.2ex}

In $\mathcal{A}$ we also have the subcategories $\Ui$, $\Ci$, $\Wi$ and $\Fi$ from Definition \ref{Ui}. Similarly, in $\cat{B}=\Ch(\mathcal{A})$ we have the subcategories $\Ui[\cat{B}]$, $\Ci[\cat{B}]$, $\Wi[\cat{B}]$ and $\Fi[\cat{B}]$. By an argument dual to the proof of Proposition~\ref{Ch-p}, one shows the following result.

\begin{proposition} 
\label{Ch-i}
Assume that $\FID(\class A)<\infty$ and let $Y=\cdots \rightarrow Y_{n+1}\rightarrow Y_{n}\rightarrow Y_{n-1}\rightarrow\cdots$ be an object in $\mathcal{B}:=\Ch(\class A)$. The following conclusions hold.
\begin{itemize}
\item [(i) ] $Y$ belongs to $\Ui[\cat{B}]$ if and only if every $Y_{n}$ belongs to $\Ui$.
\item [(ii) ] $Y$ belongs to $\Ci[\cat{B}]$ if and only if every $Y_{n}$ belongs to $\Ci$.
\item [(iii) ] $Y$ belongs to $\Wi[\cat{B}]$ if and only if $Y$ is exact and every cycle $\mathrm{Z}_{n}(Y)$ belongs to $\Wi$.
\item [(iv) ] $Y$ belongs to $\Fi[\cat{B}]$ if and only if every $Y_{n}$ belongs to $\Fi$. \qed
\end{itemize}
\end{proposition}

We can now prove the main result of this section.

\begin{theorem}
  \label{Ch}
   Let $(S,T)$ be a Sharp--Foxby adjunction on $\cat{A}$, in particular,  $\sGProj(\cat{A})$ and $\sGInj(\cat{A})$ are equivalent as triangulated categories by Corollary~\ref{cor:SF-imples-GProj-GInj}. If $\,\FPD(\cat{A})<\infty$ and $\FID(\cat{A})<\infty$, then degreewise application of $S$ and $T$ yields a Sharp--Foxby adjunction on $\cat{B}=\mathrm{Ch}(\cat{A})$, and hence $\sGProj(\cat{B})$ and $\sGInj(\cat{B})$ are equivalent as triangulated categories.
\end{theorem}

%In the proof below, we abuse notation and use the same symbols $S$ and $T$ for the functors they induce (by degreewise application) on $\cat{B}=\mathrm{Ch}(\cat{A})$.

\begin{proof}
  Write $\bar{S}$ and $\bar{T}$ for the endofunctors on $\cat{B}=\mathrm{Ch}(\cat{A})$ that are given by degreewise application of $S$ and $T$, and let $\eta$ and $\varepsilon$ be the unit and counit of the adjunction $(S,T)$~on~$\cat{A}$. It is straightforward to verify that $(\bar{S},\bar{T})$ is an adjunction on $\cat{B}$ with unit $\bar{\eta}$ and counit $\bar{\varepsilon}$ given by $(\bar{\eta}_X)_n = \eta_{X_n}$ and $(\bar{\varepsilon}_X)_n = \varepsilon_{X_n}$, where $X$ is a chain complex and $n$ is an integer.
  
  By assumption, $S$ restricts to an exact functor $S \colon \Up \to \Ui$ which maps $\Wp$ to $\Wi$; see (SF1) and (SF2) in Definition \ref{Sharp-Foxby}. It therefore follows from Propositions \ref{Ch-p} and \ref{Ch-i} that $\bar{S}$ restricts to an exact functor $\bar{S} \colon \Up[\cat{B}] \to \Ui[\cat{B}]$ which maps $\Wp[\cat{B}]$ to $\Wi[\cat{B}]$, that is, the adjunction $(\bar{S},\bar{T})$ also satisfies conditions (SF1) and (SF2). A similar argument shows that this adjunction satisfies (SF3) and (SF4) as well. By (SF5) in Definition \ref{Sharp-Foxby} we know that the unit $\eta_A \colon A \to TSA$ of $(S,T)$ is an isomorphism for $A \in \Up$. From the definition of $\bar{\eta}$ and from Proposition \ref{Ch-p} it now follows that $\bar{\eta}_X \colon X \to \bar{T}\bar{S}X$ is an isomorphism for $X \in \Up[\cat{B}]$, that is, $(\bar{S},\bar{T})$ satisfies (SF5). Similarly, $(\bar{S},\bar{T})$ also satisfies condition (SF6).
\end{proof} 

\begin{corollary}
  \label{Ch-cor}
   Let $(S,T)$ be a Sharp--Foxby adjunction on $\cat{A}$ for which $\FPD(\cat{A})<\infty$ and $\FID(\cat{A})<\infty$. Then degreewise application of $S$ and $T$ yields a Sharp--Foxby adjunction on the category $\mathrm{Ch}^2(\cat{A})$ of double complexes (also called bicomplexes) in $\cat{A}$.
\end{corollary}

\begin{proof}
  The category $\mathrm{Ch}^2(\cat{A})$ of double complexes in $\cat{A}$ is naturally identified with the ca\-te\-go\-ry $\mathrm{Ch}(\mathrm{Ch}(\cat{A}))$. Thus, the desired conclusion follows by applying Theorem~\ref{Ch} to the category $\mathrm{Ch}(\cat{A})$ (in place of $\cat{A}$). However, to do this we must first argue that the theorem's hypothesis is satisfied, i.e.~that the numbers $\FPD(\mathrm{Ch}(\cat{A}))$ and $\FID(\mathrm{Ch}(\cat{A}))$ are finite. But is immediate from (the proofs of) Propositions~\ref{Ch-p}(iii) and \ref{Ch-i}(iii) that these numbers agree with $\FPD(\cat{A})$ and $\FID(\cat{A})$, which are finite by assumption.
\end{proof}

\enlargethispage{5.1ex}

\begin{example}
\label{ex.ch}
Let $A$ be a commutative noetherian ring with a dualizing module. By Example \ref{Auslander-Bass} there exists a Sharp--Foxby adjunction on $\mathrm{Mod}(A)$. The finitistic projective/injective dimensions of $\mathrm{Mod}(A)$ are usually referred to as the finitistic projective/injective dimensions of the ring $A$, and they are denoted by $\FPD(A)$ and $\FID(A)$. These numbers are finite, indeed, one has $\FPD(A) = \operatorname{dim}\,A \geqslant \FID(A)$ by \cite[Thm.~II.(3.2.6) p.~84]{LGrMRn71} and \cite[Cor.~5.5]{HBs62}, and $\operatorname{dim}\,A$ is finite by \cite[Cor.~V.7.2]{rad}.

Theorem~\ref{Ch} and Corollary~\ref{Ch-cor} now imply that the category $\Ch(A)$ of chain complexes and the category $\Ch^2(A)$ of double complexes of $A$-modules both have Sharp--Foxby adjunctions. In particular, there are by Corollary~\ref{cor:SF-imples-GProj-GInj} equivalences of triangulated categories,
\begin{displaymath}
  \sGProj(\Ch(A)) \simeq \sGInj(\Ch(A))
  \qquad \text{and} \qquad
  \sGProj(\Ch^2(A)) \simeq \sGInj(\Ch^2(A))\;.
\end{displaymath}
\end{example}

\section{The case of quiver representations}
\label{sec:Rep}

Recall from the beginning of Section~\ref{Preliminaries} that $\cat{A}$ always denotes any bicomplete abelian category with enough projectives and enough injectives. In this section, we consider the abelian category $\rep{Q,\class A}$ of $\cat{A}$-valued representations of a quiver $Q$ and prove that, under suitable conditions, a Sharp--Foxby adjunction $(S,T)$ on $\class A$ induces a Sharp--Foxby adjunction on $\rep{Q,\class A}$ by vertexwise application of the functors $S$ and $T$.

We start by collecting a few facts about quivers that we need.

\begin{ipg}
  Let $Q$ be a quiver with vertex set $Q_0$ and arrow set $Q_1$.
  Unless anything else is men\-tioned, there will be no restrictions on $Q$; it may have infinitely many vertices, loops and/or oriented cycles, and there may be infinitely many or no arrows from one vertex to another. For an arrow $a \in Q_1$ we write $s(a)$ and $t(a)$ for its source and target. Following \cite[Sect.~3]{EE-projective-quivers}, a quiver is called \emph{left rooted}, respectively, \emph{right rooted}, if it does \textsl{not} contain any path of the form \mbox{$\cdots\to\bullet\to\bullet\to\bullet$}\,, respectively, of the form $\bullet\to\bullet\to\bullet\to\cdots$\,. Recall that for every vertex $v \in Q_0$, the evaluation functor \mbox{$\rep{Q,\class A} \to \cat{A}$} given by \mbox{$X \mapsto X(v)$} has a left adjoint $f_v$, which for $A \in \cat{A}$ is given by $f_v(A)(w) = \bigoplus_{v \to \cdots \to w}A$, where the coproduct contains one copy of $A$ for each path from $v$ to $w$ in $Q$ (including the trivial path at $v$ in the case where $w=v$). This copro\-duct might be empty, in which case $f_v(A)(w) =0$. Evidently, the counit of this adjunction $f_v(X(v)) \to X$ becomes an epimorphism in $\cat{A}$ when evaluated on $v$, and this 
explains the epimorphism in the exact sequence below, see e.g.~\cite[Sec.~3.1]{EHS-total-acyclicity-quivers}.
\begin{equation}
  \label{eq:quiver-ses}
0\longrightarrow\bigoplus\limits_{a \in Q_1}f_{t(a)}(X(s(a)))\longrightarrow\bigoplus\limits_{v\in Q_0}f_v(X(v))\longrightarrow X\longrightarrow 0\;.
\end{equation}  
Finally, for every vertex $v \in Q_0$ and every $X$ in $\rep{Q,\class A}$ we write
\begin{displaymath}
  \xymatrix{
  \displaystyle \bigoplus_{a \in Q_1 \text{ with } t(a)=v} \mspace{-30mu} X(s(a))
  \ar[r]^-{\varphi^X_v} & X(v)
  }
  \qquad \text{and} \qquad
  \xymatrix{  
    X(v) \ar[r]^-{\psi^X_v} & \displaystyle
   \mspace{-35mu} \prod_{a \in Q_1 \text{ with } s(a)=v} 
   \mspace{-30mu} X(t(a))
  }
\end{displaymath}
for the canonical morphisms.
\end{ipg}

In $\mathcal{A}$ we have the subcategories $\Up$, $\Cp$, $\Wp$ and $\Fp$ from Definition \ref{Up}. Similarly, in $\cat{B}=\rep{Q,\class A}$ we have the subcategories $\Up[\cat{B}]$, $\Cp[\cat{B}]$, $\Wp[\cat{B}]$ and $\Fp[\cat{B}]$. The following result explains the relation between all these subcategories.

\enlargethispage{6.1ex}

\begin{proposition} 
  \label{Rep-p}
   Assume that $\FPD(\class A)<\infty$ and let $Q$ be a left rooted quiver. For every object $X$ in $\mathcal{B}:=\rep{Q,\class A}$ the following conclusions hold.
\begin{itemize}
\item [(i) ] $X$ belongs to $\Up[\cat{B}]$ if and only if $X(v)$ belongs to $\Up$ for every vertex $v$ in $Q$.
\item [(ii) ] $X$ belongs to $\Cp[\cat{B}]$ if and only if $X(v)$ belongs to $\Cp$ and $\varphi^X_v$ is a monomorphism with $\coker \varphi^X_v$ in $\Cp$ for every vertex $v$ in $Q$.
\item [(iii) ] $X$ belongs to $\Wp[\cat{B}]$ if and only if $X(v)$ belongs to $\Wp$ for every vertex $v$ in $Q$.
\item [(iv) ] $X$ belongs to $\Fp[\cat{B}]$ if and only if $X(v)$ belongs to $\Fp$ for every vertex $v$ in $Q$.
\end{itemize}
\end{proposition}

\begin{proof}
(ii): This characterization of Gorenstein projective representations over a left rooted quiver is given in \cite[Thm.~3.5.1(b)]{EHS-total-acyclicity-quivers}. Their proof works in any abelian category.

(i): In view of (ii), the ``only if'' part is clear. Conversely, assume that $X(v) \in \Up$, that is, $\GpdA X(v)$ is finite, for every $v \in Q_0$. By  assumption, $\FPD(\class A)$ is finite and by Lemma~\ref{finitistic} it follows that $n=\sup\{\GpdA X(v) \,|\, v \in Q_0 \}$ belongs to $\mathbb{N}_0$. Note that if $Q$ has only finitely many vertices, then it automatically follows that $n<\infty$, and we do not need to assume that $\FPD(\class A)$ is finite. By induction on $n$ we now show that $\mathrm{Gpd}_{\cat{B}}(X) \leqslant n+1$ (and thus~$X \in \Up[\cat{B}]$). 

If $n=0$, then $X(v) \in \cat{A}$ is Gorenstein projective for every $v \in Q_0$. For every Gorenstein projective object $G \in \cat{A}$ and every vertex $v \in Q_0$, it follows easily from part (ii) that the representation $f_v(G)$ is Gorenstein projective, so the sequence (\ref{eq:quiver-ses}) yields $\mathrm{Gpd}_{\cat{B}}(X) \leqslant 1$.

Now assume that $n>0$ and take a short exact sequence $0 \to X' \to G \to X \to 0$ in $\cat{B}$ where $G$ is Gorenstein projective. Then $n'=\sup\{\GpdA X'(v) \,|\, v \in Q_0 \}$ satisfies $n' \leqslant n-1$, and it follows from the induction hypothesis that $\mathrm{Gpd}_{\cat{B}}(X') \leqslant n$. Thus $\mathrm{Gpd}_{\cat{B}}(X) \leqslant n+1$.

(iii) Dual to  \cite[Prop.~6.5]{EEGR-injective-quivers}.

(iv) Since  $\Fp=\Up$ and $\Fp[\cat{B}]=\Up[\cat{B}]$ this is just a repetition of part (i).
\end{proof}

In $\mathcal{A}$ we also have the subcategories $\Ui$, $\Ci$, $\Wi$ and $\Fi$ from Definition \ref{Ui}. Similarly, in $\cat{B}=\rep{Q,\class A}$ we have the subcategories $\Ui[\cat{B}]$, $\Ci[\cat{B}]$, $\Wi[\cat{B}]$ and $\Fi[\cat{B}]$. By an argument dual to the proof of Proposition~\ref{Rep-p}, one shows
the following result.

\begin{proposition} 
  \label{Rep-i}
   Assume that $\FID(\class A)<\infty$ and let $Q$ be a right rooted quiver. For every object $Y$ in $\mathcal{B}:=\rep{Q,\class A}$ the following conclusions hold.
\begin{itemize}
\item [(i) ] $Y$ belongs to $\Ui[\cat{B}]$ if and only if $Y(v)$ belongs to $\Ui$ for every vertex $v$ in $Q$.
\item [(ii) ] $Y$ belongs to $\Ci[\cat{B}]$ if and only if $Y(v)$ belongs to $\Ci$ and $\psi^Y_v$ is an epimorphism with $\ker \psi^Y_v$ in $\Ci$ for every vertex $v$ in $Q$.
\item [(iii) ] $Y$ belongs to $\Wi[\cat{B}]$ if and only if $Y(v)$ belongs to $\Wi$ for every vertex $v$ in $Q$.
\item [(iv) ] $Y$ belongs to $\Fi[\cat{B}]$ if and only if $Y(v)$ belongs to $\Fi$ for every vertex $v$ in $Q$. \qed
\end{itemize}
\end{proposition}

Now we prove the main result of this section.

\begin{theorem}
  \label{Rep}
   Let $(S,T)$ be a Sharp--Foxby adjunction on $\cat{A}$, in particular,  $\sGProj(\cat{A})$ and $\sGInj(\cat{A})$ are equivalent as triangulated categories by Corollary~\ref{cor:SF-imples-GProj-GInj}. If $\,\FPD(\cat{A})<\infty$ and $\FID(\cat{A})<\infty$, then vertexwise application of $S$ and $T$ yields a Sharp--Foxby adjunction on $\cat{B}=\rep{Q,\class A}$, and thus $\sGProj(\cat{B})$ and $\sGInj(\cat{B})$ are equivalent as triangulated categories.
\end{theorem}

\begin{proof}
  This follows from Propositions~\ref{Rep-p} and \ref{Rep-i} (as in the proof of Theorem~\ref{Ch}).
\end{proof}

\begin{example}
As in Examples \ref{Auslander-Bass} and \ref{ex.ch}, we let $A$ be a commutative noetherian ring with a dualizing module. Theorem~\ref{Rep} shows that for any quiver $Q$ which is both left rooted and right rooted, the category $\cat{B}=\rep{Q,\mathrm{Mod}(A)}$ admits a Sharp--Foxby adjunction and~thus $\sGProj(\cat{B})$ and $\sGInj(\cat{B})$ are equivalent as triangulated categories. 
\end{example}

\begin{remark}
  The proof of Theorem~\ref{Rep} (which relies on the proofs of Propositions~\ref{Rep-p}~and \ref{Rep-i}) reveals that one does not need to assume finiteness of $\FPD(\class A)$ and $\FID(\class A)$ in the case where the quiver $Q$ has only finitely many vertices. 
\end{remark}

\enlargethispage{7.5ex}

\def\soft#1{\leavevmode\setbox0=\hbox{h}\dimen7=\ht0\advance \dimen7
  by-1ex\relax\if t#1\relax\rlap{\raise.6\dimen7
  \hbox{\kern.3ex\char'47}}#1\relax\else\if T#1\relax
  \rlap{\raise.5\dimen7\hbox{\kern1.3ex\char'47}}#1\relax \else\if
  d#1\relax\rlap{\raise.5\dimen7\hbox{\kern.9ex \char'47}}#1\relax\else\if
  D#1\relax\rlap{\raise.5\dimen7 \hbox{\kern1.4ex\char'47}}#1\relax\else\if
  l#1\relax \rlap{\raise.5\dimen7\hbox{\kern.4ex\char'47}}#1\relax \else\if
  L#1\relax\rlap{\raise.5\dimen7\hbox{\kern.7ex
  \char'47}}#1\relax\else\message{accent \string\soft \space #1 not
  defined!}#1\relax\fi\fi\fi\fi\fi\fi} \def\cprime{$'$}
  \providecommand{\arxiv}[2][AC]{\mbox{\href{http://arxiv.org/abs/#2}{\sf
  arXiv:#2 [math.#1]}}}
  \providecommand{\oldarxiv}[2][AC]{\mbox{\href{http://arxiv.org/abs/math/#2}{\sf
  arXiv:math/#2
  [math.#1]}}}\providecommand{\MR}[1]{\mbox{\href{http://www.ams.org/mathscinet-getitem?mr=#1}{#1}}}
  \renewcommand{\MR}[1]{\mbox{\href{http://www.ams.org/mathscinet-getitem?mr=#1}{#1}}}
\providecommand{\bysame}{\leavevmode\hbox to3em{\hrulefill}\thinspace}
\providecommand{\MR}{\relax\ifhmode\unskip\space\fi MR }
\providecommand{\MRhref}[2]{%
  \href{http://www.ams.org/mathscinet-getitem?mr=#1}{#2}
}
\providecommand{\href}[2]{#2}

\enlargethispage{6.5ex}

\end{document}